\documentclass{amsart}

\usepackage{amsthm, amssymb}

\newtheorem{proposition}{Proposition}

\newtheorem{theorem}{Theorem}
\newtheorem{corollary}{Corollary}
\newtheorem{remark}{Remark}

\begin{document}
\title[A note on transverse divergence and taut Riemannian foliations]{A note on transverse divergence and taut Riemannian foliations}

\author[Vladimir Slesar]{Vladimir Slesar}

\date{}

\abstract
 In this note we give a characterization of taut Riemannian foliations using the transverse divergence. This result turns out to be a convenient tool in the case of some standard examples. Furthermore, we show that a classical tautness result of Haefliger can be obtained in our particular setting as a straightforward consequence. In the final part of the paper we obtain a tautness characterization for transversally oriented foliations with dense leaves.
\endabstract

\maketitle

\begin{center}
\end{center}

\section{Introduction}
\label{intro}
Throughout this paper we consider a $C^\infty $ foliated manifold
$\left( M,\mathcal{F}\right) $, endowed with a Riemannian metric $g$. A particular
class of foliations with remarkable dynamical and geometric properties is
represented by the taut foliations. A foliation is called \emph{taut} if
there exists a metric on $M$ such that every leaf of $\mathcal{F}$ is a
minimal submanifold of the ambient manifold $M$. The work of Sullivan \cite
{Sull} and later of Haefliger \cite{Hae} on compact foliated manifolds offered a
characterization of these aspects, showing that tautness is
basically related to the transverse geometry of the foliation.

An usual additional condition for a foliation is the existence of a
\emph{bundle-like} metric; with respect to this metric the foliation can be locally
identified with a Riemannian submersion \cite{Re}. This particular class of foliations are called
\emph{Riemannian foliations}.  Examples of this type occur for instance in the theory of locally conformally K\"ahler manifolds and their submanifolds (see \textrm{e.g.} \cite[page 44]{Drag-Ornea}, \cite{Vilcu}), in the case of Sasakian manifolds \cite[Subsection 1.1]{Bo-Ga}, while the
characteristic leafwise holonomy invariance is relevant from the point of view of
Yang-Mills equations (see \textrm{e.g.} \cite{I-S-V-V} for further references related to particle physics).

For Riemannian foliations the tautness properties has been intensively studied;
for a chronological and more detailed presentation we
refer to \cite{Prieto-Sar-Wolak}. Remarkably, the tautness depends on the topology
of the manifold, and a characterization in terms of the basic cohomology is
possible. More exactly, if the manifold $M$ is closed (i.e. compact and
without boundary), then it can be shown that the foliation is taut if and
only if a certain basic cohomology class $\xi (\mathcal{F})$-called
\emph{\'Alvarez class}, is trivial \cite{Al}. If, furthermore, the foliation is
transversally orientable, then it is taut if and only if the top dimensional
cohomology group does not vanish \cite{Al,Mas}.

As a consequence, in many papers the study of taut
foliations is based on the topological and cohomological aspects. Thus,
when investigating particular examples, beside a direct study of the basic
cohomology groups (see \textrm{e.g.} \cite{Car}), a tautness result via some adapted vanishing techniques can be expected at most in the case of foliations with positively defined transverse curvature operator (see \cite{Al,Mi-Ru-To}).

In the present note we take a different approach, establishing a new
characterization of taut Riemannian foliations using the \emph{transverse divergence
operator}. From the dynamical point of view, a foliation can be regarded as a
generalization of a differential manifold (the classical setting is
obtained in the absolute case when the manifold is foliated by points). On a
closed Riemannian manifold we cannot construct vector fields with positive,
non-vanishing divergence, as a trivial consequence of the Green theorem.
Considering the transverse geometry of the foliations as a generalization for
the geometry of the manifolds and the transverse divergence as an extension
of the classical differential operator, we prove in this note
that taut Riemannian foliations are precisely the extended
framework of Riemannian manifolds where the above classical fact still
holds.

As this result does not rely on basic cohomology and Bochner technique, we show that it can be applied
on classical examples with non-positive transverse curvature operator.

It also turns out that some classical statements from the theory of taut Riemannian
foliations can be derived as straightforward consequences of our analytical
characterization. In particular, a tautness result of Haefliger
\cite{Hae}, asserting that the existence of a closed transverse submanifold
ensures the tautness is proved for our particular framework, revealing an
interesting link between this result and the Green theorem. Also,
an important extension to the finite Riemannian covering spaces of the
tautness properties proved in \cite{Al} can be alternatively derived using
our arguments in a simpler manner.

In the final we give a tautness characterization for transversely oriented Riemannian foliations with
dense leaves using the transverse volume form.

The paper is organized as follows: in the second section we define the main
geometric objects we use in this note and briefly state the relevant
previous results. The central result of the note is proved in the third
section, while in the last part we present the above straightforward applications.

\section{Preliminaries}
\label{sec:1}
We start out this section considering a differential manifold $M$ endowed with a Riemannian metric structure $g$
and with a non-necessary integrable distribution $D$ of dimension $d$.  In the following we use
the classical musical isomorphisms $\sharp $ and $\flat$. Let
$\left\{ {E}_i\right\} _{1\le i\le d}$ and $\left\{ {E}_i^{\flat }\right\} _{1\le
i\le d}$ be a local orthonormal frame and coframe for $D$ with respect to the metric $g$.

The action of the \emph{divergence operator associated to $D$} on a vector field ${v}\in \Gamma \left(
TM\right) $ is defined using the Levi-Civita connection $\nabla $
\begin{equation}
\mathrm{div}^D {v}:=\sum\limits_{i=1}^dg\left( \nabla _{{E}_i}{v},{E}_i\right)\;.
\label{divergenta_general}
\end{equation}

In the following, we briefly present the interplay between the divergence
operator and a (locally defined) volume form $\nu _D$ of $D$ induced in the canonical way by the metric $g$.
According to \cite{Al-Ko} (see also \cite[Chapter 4]{To}), if $\mathcal{L}_{{v}}$ is the Lie derivative
along the vector field $v$, we consider the restriction
$\theta _{v}:={\mathcal{L}_{v}}_{\mid D}$; for instance, if $1\le i,j\le d$, then
\begin{eqnarray*}
{\mathcal{L}_{v}} {E}_i^{\flat }( E_j) &=&v( {E}_i^{\flat }({E}_j) ) -{E}_i^{\flat }\left( [{v},{E}_j]\right) \\
&=&-g( [{v},{E}_j],{E}_i)\;,
\end{eqnarray*}
and, consequently
\[
\theta_{{v}} {E}_i^{\flat }=-\sum_{j=1}^dg\left( [{v},{E}_j],{E}_i\right) {E}_j^{\flat }\;.
\]

Now, taking $\nu _D={E}_1^{\flat }\wedge ..\wedge {E}_d^{\flat }$, as the
Levi-Civita connection has vanishing torsion, we get
\begin{eqnarray}
\theta _{{v}}\nu _D &=&\theta _{{v}} {E}_1^{\flat }\wedge ..\wedge E_d^{\flat}+..
+ {E}_1^{\flat }\wedge ..\wedge \theta _{{v}} {E}_d^{\flat }  \nonumber \\
\ &=&-\sum_{j=1}^dg\left( [{v}, {E}_j], {E}_j\right) \nu _D  \nonumber \\
\ &=&\mathrm{div}^D {v}\cdot \nu _D \;.  \label{div_vol}
\end{eqnarray}

\begin{remark}
In the classical case when $D\equiv TM$, we get the well-known equivalent definition for the divergence operator
(see e.g. \cite[page 140]{Marsden-Ratiu})
\[
{\mathcal{L}_{{v}}}\nu_{TM}=\mathrm{div}^{TM} {v}\cdot \nu_{TM}\;.
\]
\end{remark}

Considering the complementary distribution $D^{\perp }$ with respect to the
metric $g$ and taking the canonical projections $\pi _D$ and $\pi _{D^{\perp}}$, if ${v}\in D^{\perp }$, then it is easy to see that
\begin{equation}
\theta _{{v}}\nu _D =-g( {v},\pi _{D^{\perp }}(\sum_{j=1}^d \nabla _{{E}_j}{E}_j)) \nu _D\;.
\label{div_vol1}
\end{equation}

Using the musical isomorphism $\sharp$ we can define the \emph{mean curvature vector field} ${{\kappa}}_{D} ^{\sharp }:=
\pi_{D^{\perp }}( \sum_{j=1}^d\nabla _{{E}_j}{E}_j) $, the \emph{mean
curvature form }${{\kappa}}_{D} $ being then subject to the condition
${{\kappa}}_{D}({U})=g({{\kappa}}_{D}^{\sharp },{U}) $, for any vector field ${U}$ (see \textrm{e.g.} \cite{Wal}).

After we made the above considerations in a larger setting, from now on the framework of this paper will be
represented by a smooth, closed Riemannian foliation $\left(M, \mathcal{F},g \right)$.
The \emph{leafwise} distribution tangent to leaves will be
denoted by $T\mathcal{F}$; a corresponding \emph{transverse} distribution
$Q=T\mathcal{F}^{\perp }\simeq TM/T\mathcal{F}$ is obtained. We assume
$\dim M=n$, $\dim T\mathcal{F}=p$ and $\dim Q=q$, with $p+q=n$.

A first consequence is the splitting of the tangent and the cotangent vector
bundles associated with $M$
\begin{eqnarray*}
TM &=&Q\oplus T\mathcal{F}, \\
TM^{*} &=&Q^{*}\oplus T\mathcal{F}^{*}.
\end{eqnarray*}

For local investigation of the transverse geometry of the foliated manifold
we will use local orthonormal basis $\left\{ {E}_i,{F}_a\right\} $ defined on a
neighborhood of a point $p\in M$; $\left\{ {E}_i\right\} _{1\le i\le q}$
will span the distribution $Q$ while $\left\{ {F}_a\right\} _{1\le a\le p}$
will span the distribution $T\mathcal{F}$. It is also convenient to employ
\emph{basic (projectable) local vector fields} $\{{E}_i\}_{1\le i\le q}$, defined by the condition \cite[page 4]{To}
\begin{equation}
\pi _Q\left([{U},{E}_i]\right)=0  \label{basic vector field}
\end{equation}
for any ${U}\in \Gamma (T\mathcal{F})$.

Another important transverse geometric object is represented by the
\emph{basic de Rham complex}. It is defined as a restriction of the set of
classical differential forms $\Omega \left( M\right) $ \cite[page 33]{To}

\[
\Omega _b\left( M\right) :=\left\{ \omega \in \Omega \left( M\right) \mid
\iota _{{U}}\omega =\mathcal{L}_{{U}}\omega =0\right\}\;,
\]
where ${U}$ is again an arbitrary leafwise vector field, while $\iota $ stands
for interior product. The corresponding \emph{basic de Rham derivative} $d_b$
comes as a restriction of the classical de Rham derivative, $d_b:=d_{\mid
\Omega _b\left( M\right) }$. It is also possible to define as well the adjoint operator,
namely the \emph{basic co-derivative} $\delta _b$ (see \textrm{e.g.} \cite{Al}).
The \emph{basic cohomology groups} of the basic de Rham
complex $\left( \Omega _b\left( M\right) ,d_b\right) $ are defined in a
classical way \cite[page 33]{To}) and are denoted by $H_b^i(\mathcal{F)}$, with $1\le i\le q$.

One differential form of particular importance (which may not be necessarily
a basic differential form) is the mean curvature form associated to the
leafwise distribution $T\mathcal{F}$. We denote it by ${{\kappa}}$, in accordance with \cite{Al,Rum}.
As above, it is subject to the relation
${{\kappa}} ^{\sharp }:=\pi _Q\left( \sum_{a=1}^p\nabla_{{F}_a}{F}_a\right)$.

If the distribution $D\equiv Q$, then using (\ref{divergenta_general}) we obtain the
\emph{transverse divergence operator }$\mathrm{div}^Q$ \cite[page 42]{To}. We state here
the following useful relation
\begin{equation}
\mathrm{div}^Q {v}=\mathrm{div}\,{v}+g\left( {v},{{\kappa}} ^{\sharp }\right)\;.
\label{diverg tondeur}
\end{equation}
which holds for any vector field ${v}\in \Gamma (TM)$.

A deeper investigation of the mean curvature form on Riemannian foliations is
presented in \cite{Al}; the author shows that we have the orthogonal
decomposition
\[
\Omega \left( M\right) =\Omega _b\left( M\right) \oplus \Omega _b\left(
M\right) ^{\perp }\;,
\]
with respect to the $C^\infty -$Fr\'echet topology. Thus, on any Riemannian
foliation, the mean curvature form can be decomposed as the sum
\[
{{\kappa}} ={{\kappa}} _b+{{\kappa}}_o \;,
\]
where ${{\kappa}} _b\in \Omega _b\left( M\right) $ is the $\emph{basic}$
\emph{component of the mean curvature}, ${{\kappa}} _o$ being the orthogonal
complement. From now on we denote ${\tau} :={{\kappa}} _b^{\sharp }$.

\begin{remark}
A fundamental property of ${{\kappa}} _b$ is that it is a closed differential basic form,
the corresponding basic cohomology class being denoted by $\xi (\mathcal{F})$;
this cohomology class (called \emph{\'Alvarez class} \cite{Noz,Noz-Prieto})
is independent on the metric $g$.
\end{remark}

We state now a remarkable characterization of the taut Riemannian
foliation using basic cohomology obtained in \cite{Al}.

\begin{theorem}[\'Alvarez L\'opez \cite{Al}]
\label{theorem alvarez}A Riemannian foliation $F$ on a compact manifold is
taut if and only if $\xi (\mathcal{F})=0$. Moreover, when $\mathcal{F}$ is
transversely oriented, the foliation is taut if and only if $H_b^q(\mathcal{F})\not =0$.
\end{theorem}

\section{A new characterization of taut Riemannian foliations}
\label{sec:2}
In the following, we show that it is possible to obtain a characterization of the non-taut
Riemannian foliation of arbitrary codimension using the transverse divergence.

\begin{theorem}
\label{characterization}Let $(M,g,\mathcal{F})$ be a Riemannian foliation defined on a closed
manifold $M$. Then the foliation is non-taut if and only if there is a basic
vector field $v$ on $M$ such that $\mathrm{div}^Q {v}\ge 0$ and $\mathrm{div}^Q {v}>0$
at some point, where $\mathrm{div}^Q$ is the transverse divergence
operator associated to the metric $g$.
\end{theorem}

\textit{Proof}. We prove first that the above condition regarding the transverse divergence implies that the foliation is non-taut.
As the metric $g$ is bundle-like, it is easy to see that $\mathrm{div}^Q\,{v}=\mathrm{div}^Q\,\pi_Q({v})$, so we can assume that ${v}$ is also perpendicular to the leaves at any point. We consider the \emph{transverse Green formula} (see \textrm{e.g.}
\cite[page 42]{To}):
\[
\int_M\mathrm{div}^Q {v}\,d\mu =\int_Mg( {v},k^{\sharp }) d\mu \;,
\]
the integrals being taken with respect to the measure $d\mu $ induced by the
metric $g$. We point out that the above formula arises from the relation
(\ref{diverg tondeur}) and the theorem of Green on closed, non-necessarily
orientable Riemannian manifolds (see \textrm{e.g.} \cite[page 157]{Po}).

We assume that the foliation is taut and, in accordance with \cite{Al,Mason}, it is possible
to choose a bundle-like metric $g^{\prime }$ such that
the new mean curvature 1-form vanishes, ${{\kappa}} _b^{\prime }={{\kappa}} ^{\prime }=0$, the
transverse part of the metric remaining intact. We consider the relation
\[
\mathrm{div}^Q {v}=\sum_{i=1}^q\iota _{{E_i}}\nabla _{{E_i}}v^{\flat } \;,
\]
the musical isomorphism being considered with respect to the initial metric $g$.
It turns out that $v^{\flat }$ is a basic 1-form on $\mathcal{F}$. Now, taking the basic local
vector fields $\left\{ E_i^{\prime }\right\} _{1\le i\le q}$ as a transverse
orthonormal frame with respect to the new metric $g^{\prime }$ spanning the
new orthogonal complement $Q^{\prime }$ of $T\mathcal{F}$, as the transverse
part of the metric is invariant, we see that
\[
\sum_{i=1}^q\iota _{{E}_i}\nabla _{{E}_i}{v}^{\flat }=
\sum_{i=1}^q\iota_{{E}_i^{\prime }}\nabla _{{E}_i^{\prime }}^{\prime }{v}^{\flat} \;,
\]
and

\[
\mathrm{div}^Q {v}=\mathrm{div}^{Q^{\prime }} {v}^{\prime } \;,
\]
where we denote ${v}^{\prime }:=( {v}^{\flat }) ^{\sharp ^{\prime }}$,
this time the operator $\sharp ^{\prime }$ being constructed with respect
to the bundle-like metric $g^{\prime}$; note that ${v}^{\prime }$ may not
equal ${v}$, as $Q^{\prime }$ may not coincide with $Q$. In fact, regarding
locally the Riemannian foliation as a Riemannian submersion, if we consider
a local transverse manifold $T$, as all geometrical objects are
projectable, then both operators equal $-\delta_T$, where $\delta_T$ is
the classical de Rham coderivative on $T$ (see \textrm{e.g.} \cite{Al}). Integrating
the real function $\mathrm{div}^Q {v}$ over the closed manifold $M$ with
respect to the modified metric, we have

\begin{eqnarray*}
\int_M \mathrm{div}^{Q} {v}\,d\mu ^{\prime } &=&\int_M \mathrm{div}^{Q^{\prime}} {v}^{\prime }d\mu ^{\prime } \\
&=&\int_Mg^{\prime }( {v}^{\prime },{{\kappa}}^{\prime \sharp^{\prime}}) d\mu ^{\prime} \;.
\end{eqnarray*}

As we assumed ${{\kappa}}^{\prime }=0$, we get also ${{\kappa}}^{\prime \sharp^{\prime }}=0$,
and $\int_M\mathrm{div}^{Q} {v}\,d\mu ^{\prime }=0$. Now the
contradiction comes from the positivity of $\mathrm{div}^{Q} v$.

In the sequel we prove the converse statement. Let $(M,g,\mathcal{F})$ be a non-taut Riemannian
foliation. According to \cite{Mason}, we can choose the metric $g$ such that the transverse part
remains intact, and with respect to the new metric $g^{\prime }$, the mean curvature $k^{\prime }$ is basic and harmonic.
Then, in accordance with \cite{Al},
\begin{eqnarray*}
\delta _b{{\kappa}} ^{\prime } &=&-\sum_{i=1}^q\iota _{{E}^{\prime}_i}\nabla^{\prime} _{{E}^{\prime}_i}{{\kappa}}^{\prime }
+\iota _{{\tau} ^{\prime }}{{\kappa}} ^{\prime } \\
&=&0 \;,
\end{eqnarray*}
where $\iota $ stands for interior multiplication and ${\tau} ^{\prime }$ is
defined like ${\tau} $ for $(M,g^{\prime },\mathcal{F})$.

On the other side,
\begin{equation}
\sum_{i=1}^q\iota _{{E}^{\prime}_i}\nabla^{\prime} _{{E}^{\prime}_i}{{\kappa}}^{\prime }   =\mathrm{div}^Q {\tau}^ {\prime}  \,\,\,\mbox{and}\,\,\,
\iota _{{\tau} ^{\prime }}{{\kappa}} ^{\prime } =\left\| {\tau} ^{\prime }\right\|^2 \;, \nonumber
\end{equation}
and thus we get
\[
\mathrm{div}^{Q^{\prime }}{\tau} ^{\prime }=\left\| {\tau} ^{\prime }\right\|^2 \;.
\]
As the foliation is non-taut, then ${\tau} ^{\prime }\not \equiv 0$, so
$\mathrm{div}^{Q^{\prime }} {\tau} ^{\prime }\ge 0$ and $\mathrm{div}^{Q^{\prime}} {\tau} ^{\prime }>0$ at some point. Now, we take
\begin{eqnarray*}
{\tau} _{\xi (\mathcal{F}),g} &:=&( ( {\tau} ^{\prime }) ^{\flat^{\prime }}) ^{\sharp } \\
&=&( {{\kappa}} ^{\prime }) ^{\sharp } \;.
\end{eqnarray*}
We remark that the transverse vector field ${\tau} _{\xi (\mathcal{F}),g}$
is determined only by the basic cohomology class $\xi (\mathcal{F})$ of the
foliation and the initial bundle-like metric $g$. As before, we obtain
\[
\mathrm{div}^Q {\tau} _{\xi (\mathcal{F}),g}=\mathrm{div}^{Q^{\prime }} {\tau}^{\prime } \;.
\]

We notice that ${\tau} ^{\prime }$ and ${\tau} _{\xi (\mathcal{F}),g}$ are basic
vector fields, while $\mathrm{div}^Q {\tau} _{\xi (\mathcal{F}),g}$ is a basic
real function, which finishes the proof of the theorem.
\qed

\begin{corollary}
\label{taut characterization}A Riemannian foliation on a closed manifold is taut if and only if
for any basic vector field ${v}$ we have $\mathrm{div}^Q {v}=0$, or there are at
least two points $p_1$ and $p_2$ such that $\mathrm{div}_{p_1}^Q {v}> 0$ and
$\mathrm{div}_{p_2}^Q {v}< 0$.
\end{corollary}

\begin{remark}
\label{Rem transverse geometry}If we assume the transverse orientation, then another result that holds only on taut Riemannian foliation is represented by the Poincar\'e duality (see \textrm{e.g.} \cite{Al,Mas}). This, together with the
above result lead us to regard taut foliations as the maximal extension of closed Riemannian manifolds when some classical features are still valid.
\end{remark}

\section{Several applications}
\label{sec:3}
\subsection{A new tautness criterion}

In the following we test the above result on some standard examples of
Riemannian foliations, showing that it turns out to be a convenient tautness criterion.

\textbf{Example 1.} We consider a Riemannian flow which is classical in the theory of foliations.
We take a matrix $\mathrm{A}\in \mathrm{SL}(2,\mathbb{Z})$, with trace condition $\mathrm{Tr} \mathrm{A}>2$.
If $\lambda _1$, $\lambda _2$ are the eigenvalues, it is easy to see that
\[
0<\lambda _1\cdot \lambda _2=1,\,\lambda _i\neq 1,\,\lambda _i>0,
\]
for $1\le i\le 2$. Let ${V}_1$, ${V}_2$ be the corresponding orthonormal
eigenvectors. Then, one can define a semi-direct product
$H=\mathbb{R}\times_\varphi \mathbb{R}^2$, where $\left\{ \varphi _t\right\} _{t\in \mathbb{R}}$
is defined as
\[
\varphi _t(x)=\mathrm{A}^tx,\ \forall x\in \mathbb{R}^2.
\]

We get a Lie group $H$, such that if $p_1$, $p_2\in H$, $p_1=(t_1,x_1)$, and
$p_2=(t_2,x_2)$, then the multiplication will be defined in the following way:
\[
p_1\cdot p_2=(t_1+t_2,\mathrm{A}^{t_1}x_2+x_1) \;.
\]

Now, if $x_1=\alpha _1 {V}_1+\beta _1 {V}_2$, $x_2=\alpha _2 {V}_1+\beta _2 {V}_2$, as $\mathbb{R}\times\mathbb{R} {V}_1\times\mathbb{R} {V}_2\equiv\mathbb{R}^3$,
we can consider\ $p_1=(t_1,\alpha_1,\beta _1)$, $p_2=(t_2,\alpha _2,\beta _2)$, and the above multiplication can be written
as:
\[
p_1\cdot p_2=(t_1+t_2,\lambda _1^{t_1}\alpha _2+\alpha _1,\lambda
_2^{t_1}\beta _2+\beta _1) \;.
\]

As $\mathbb{R}\times\mathbb{R} {V}_1\times\mathbb{R} {V}_2\equiv\mathbb{R}^2\times \mathbb{R}$,
the first term becomes isomorphic with the
\emph{orientation preserving affine group } ${GA}$, defined as

\[
{GA}:=\left\{ \left(
\begin{array}{cc}
a & b \\
0 & 1
\end{array}
\right) \right\}_{a\in \mathbb{R}_{+}, b\in \mathbb{R}} \;,
\]
the isomorphism being expressed by the correspondence \cite{Car}:
\[
\left(
\begin{array}{cc}
a & b \\
0 & 1
\end{array}
\right) \longleftrightarrow (\log _{\lambda _1}a,b) \;.
\]
Is it now easy to see that $\Gamma :=\mathbb{Z\times }_\varphi \mathbb{Z}^2$
represents in fact a discrete and uniform subgroup.

We introduce a left invariant metric on $H$. With respect to the
above identification of $\mathbb{R}^3$, at $e=(0,0,0)$ we take the vectors
$\left( {E}_1\right) _e=(1,0,0)$, $\left( {E}_2\right) _e=(0,1,0)$ and $\left( {E}_3\right) _e=(0,0,1)$.

For any $p=(t_p,\alpha _p,\beta _p)\in H$, let $l_p$ be the left
multiplication by $p$, and $L_p:=dl_p$ the differential application. If we
consider the curve \ $c:\mathbb{R}\rightarrow H$ given by $c(s):=s\left(
{E}_1\right) _e$, then $\left( {E}_1\right) _e=\frac{dc}{ds}\mid _{s=0}$ and, as
a consequence, we define:
\begin{eqnarray*}
\left( {E}_1\right)_p &=&L_p(\left( {E}_1\right) _e) \\
\ &=&\frac{dc}{ds}\mid_{s=0}\left(
(t_p+s,\lambda_1^{t_p}0+\alpha_p,\lambda_2^{t_p}0+\beta_p)\right) \\
\ &=&(1,0,0)_p \;.
\end{eqnarray*}

Analogously, we consider:
\begin{eqnarray*}
\left( {E}_2\right) _p &=&\lambda _1^{t_p}(0,1,0)_p \;, \\
\left( {E}_3\right) _p &=&\lambda _2^{t_p}(0,0,1)_p \;.
\end{eqnarray*}

Now the metric tensor $g_H$ will be chosen such that $\left\{ \left(
{E}_1\right) _p,\,\left( {E}_2\right) _p,\,\left( {E}_3\right) _p\right\} $ will
be an orthonormal basis at each point $p=(t_p,\alpha _p,\beta _p)\in H$. The
resulting metric tensor will be a \emph{warped} metric tensor, and being a
left invariant metric, $g_H$ will be also $\Gamma $-invariant. Consequently,
we can consider the metric $g$ induced on the Lie group $\mathrm{T}^3_{\mathrm{A}}=\Gamma\backslash H$, which becomes a $GA-$foliated manifold when considering the
flow induced by ${E}_3$ \cite{Car}.

We compute the Lie brackets
\begin{eqnarray*}
\lbrack {E}_1, {E}_2 \rbrack &=&\ln \lambda _1 {E}_2\;, \\
\lbrack {E}_1, {E}_3 \rbrack &=&\ln \lambda _2 {E}_3\;, \\
\lbrack {E}_2, {E}_3 \rbrack &=&0\;.
\end{eqnarray*}

For the Cartan and Christoffel coefficients
\begin{eqnarray*}
C_{ij}^k {E}_k &=&[{E}_i,{E}_j]\:, \\
\Gamma _{ij}^k &=&g( \nabla _{{E}_i}{E}_j,{E}_k )\: ,
\end{eqnarray*}
with $1\le i,j,k\le 3$, using the Koszul formula we get
\[
\Gamma _{ij}^k=\frac 12\left( C_{ij}^k+C_{ki}^j+C_{kj}^i\right) \;.
\]
We can compute now the necessary coefficients
\begin{eqnarray*}
&\Gamma _{11}^1 =\Gamma _{33}^2=\Gamma _{12}^1=0, \\
&\Gamma _{33}^1 =\ln \lambda _{2,}\,\,\Gamma _{21}^2=-\ln \lambda _1 \;.
\end{eqnarray*}
As a consequence we obtain the mean curvature vector field
\begin{eqnarray*}
\tau &=&\Gamma _{33}^1 {E}_1+\Gamma _{33}^2{E}_2 \\
&=&\ln \lambda _2{E}_{1} \;.
\end{eqnarray*}
Now, we take ${v}:={\tau} $, and obtain
\begin{eqnarray*}
\mathrm{div}^Q {\tau} &=&\ln \lambda _2\Gamma _{11}^1+\ln \lambda _2\Gamma_{21}^2 \\
&=&-\ln \lambda _2 \cdot \ln \lambda _1 \\
&=&\left( \ln \lambda _1\right) ^2>0 \;.
\end{eqnarray*}
From Theorem \ref{characterization} we obtain that the flow $T_{\mathrm{A}}^3$ is
non-taut.

\begin{remark}
The same conclusion can be reached using a direct study of the basic the Rham cohomology
 \cite{Car}. On the other side, it is also easy to see that for the
above $GA-$foliated manifold the transverse curvature
operator is non-positive, so the basic cohomology cannot be investigated
using vanishing results or other classical analytical method (see \textrm{e.g.} \cite{Mi-Ru-To,Sl}).
\end{remark}

\begin{remark}
A natural question is if on Riemannian non-taut foliations all globally
defined basic vector fields share the above property regarding the transverse divergence. We notice that if ${v}$ is basic and $v\perp {{\kappa}} ^{\sharp }$ at any point, then from (\ref{diverg tondeur})
we obtain that $\mathrm{div}^Q {v}=\mathrm{div}{v}$, and we get either $\mathrm{div}^Q {v} \equiv 0$, or
this function has positive and negative values on the manifold $M$. As a direct consequence
${\tau} _{\xi (\mathcal{F}),g}$ is not perpendicular on ${{\kappa}}^{\sharp }$ at any point. For instance, in the
above example, considering the basic vector field $E_2$, we can compute $\mathrm{div}^Q {E}_2=\Gamma _{12}^1=0$.
\end{remark}

\textbf{Example 2.} We can obtain another example increasing the transversal
dimension of the foliations; namely, we consider a matrix ${A}\in \mathrm{SL}(3,\mathbb{Z})$, for instance \cite{Do}

\[
{A}=\left(
\begin{array}{rrr}
2 & 0 & -1 \\
0 & 3 & -1 \\
-1 & -1 & 1
\end{array}
\right) .
\]
If $\left\{\lambda_i\right\}_{1\le i\le 3}$, are the eigenvalues and
$\left\{{V}_i\right\}_{_{1\le i\le 3}}$ the corresponding eigenvectors, then
computing the sign of the characteristic polynomial of $\mathrm{A}$ we obtain that
\[
0<\lambda_1<1,\,\,2<\lambda_2<3,\,\,3<\lambda_2<4 \;,
\]
and, as in the previous example

\[
\lambda _1\cdot \lambda _2\cdot \lambda _3=1,\,\,\lambda _i\neq
1,\,\,\lambda _i>0 \;,
\]
for $1\le i\le 3$. Introducing the bundle-like metric in a similar manner as
above, for ${v}:={\tau} $ we finally end up with
\begin{eqnarray*}
\mathrm{div}^{Q} {\tau} &=&-\ln \lambda _2\left( \ln \lambda _1+\ln \lambda_3\right) \\
\ &=&\left( \ln \lambda _2\right) ^2>0 \;,
\end{eqnarray*}
which shows that the Riemannian foliation is non-taut.

\begin{remark}
For a proof using cohomological arguments see \cite{Do}.
\end{remark}

Finally, we give an application of Corollary \ref{taut characterization}.

\textbf{Example 3}. We may consider the torus $T^2:=\mathbb{R}^2/\mathbb{Z}^2 $
with the metric $g=$ $e^{2f(y)}dx^2+dy^2$, for some periodic function $f$;
as consequence, $\left\{ \partial _y,e^{-f(y)}\partial _x\right\} $
represent an orthonormal basis at any point, $Q=\mathrm{span}\{\partial _y\}$,
$T\mathcal{F}=\mathrm{span}\{e^{-f(y)}\partial _x\}$. Then any $C^\infty $
transversal vector field ${v}$ will be written as
\[
{v}=\varphi (y)\partial _y \;,
\]
$\varphi (y)$ being again a periodic $C^\infty $ function. Then, the
transverse divergence is
\begin{eqnarray*}
\mathrm{div}^Q {v} &=& g\left( \nabla _{\partial _y}\varphi (y)\partial_y,\partial _y \right) \\
&=&\frac{\partial \varphi (y)}{\partial y} \;.
\end{eqnarray*}
Now, we have
\[
\varphi (0)=\varphi (1).
\]
If $\varphi (y)$ is constant with respect to $y$, then $\mathrm{div}^Q {v}=0$;
otherwise we will have $\frac{\partial \varphi (y)}{\partial y}>0$ at some
points and $\frac{\partial \varphi (y)}{\partial y}<0$ at other points. The tautness of the foliation
is implied by Corollary \ref{taut characterization}.

\subsection{Other applications: new proofs for old results}
As pointed out in the introductory section, from the characterization
theorem we can obtain a simpler proof of some well-known results.

In the following, we refer first of all to a classical tautness result of
Haefliger concerning foliations with arbitrary metric. The author gives a tautness
characterization using the holonomy pseudogroup, and this characterization
is trivially satisfied when there is a representative of the holonomy
pseudogroup defined on a closed transverse manifold \cite[Corollary 2]{Hae}. In the particular case of foliations with
bundle-like metric we can project the transverse divergence operator on the closed transversal. Consequently,
we give a short proof of this result applying the Green theorem.

\begin{proposition}[Haefliger \cite{Hae}]
If a Riemannian foliation has a closed transversal which cuts all leaves then it is taut.
\end{proposition}

\textit{Proof}. We consider a closed transverse manifold $T $ and suppose by
contradiction that the foliation is non-taut. If at some point $p\in M$ the manifold
$T $ meets a leaf, then at that point it generates a subspace $Q_p$ of
the tangent space $T_pM$ which is a complement of $T_p\mathcal{F}$. It is
easy to see that $Q_p$ can be determined by a unique retraction
$r_p:T_pM\rightarrow T_p\mathcal{F}$, with $\ker r_p=Q_p$. Considering now
the vector bundle constructed over $M$ using these mathematical objects, we
can extend the complementary distribution $Q$ from the closed subset $T
$ to the whole manifold $M$ as a classical extension of a section in a fibre
bundle (see \textrm{e.g.} \cite[Theorem 5.7]{Ko-No}).

In a similar way as in \cite{Al, Do}, requesting that $Q_p\perp T_p\mathcal{F}$
at each $p$, we can obtain a new metric which keeps the transverse and the
leafwise part of the initial metric, but this time $T $ is a
transverse orthogonal manifold. The foliation is still non-taut. According
to the characterization, there will be some basic vector field ${v}$, which can be also
made perpendicular to the leaves, such that $\mathrm{div}^{Q} {v}\ge 0$ and $\mathrm{div}^{Q} {v}>0$ at
some point. At any point where $T $ meets a leaf $\mathrm{div}^{Q}v$
in fact equals the classical divergence operator $\mathrm{div}^T {v}$ on $T $;
as $T $ meets all the leaves and the
function $\mathrm{div}^{Q} {v}$ is constant along leaves, we will
have $\mathrm{div}^T {v}\ge 0$ and $\mathrm{div}^T {v}>0$ at some
point on $T $. This fact is impossible considering the classical
Green theorem, this time on the closed manifold $T $.
\qed

As pointed out in \cite[Lemma 6.3]{Al}, the tautness of a Riemannian foliations is a property
which can be lifted on finite Riemannian coverings. As we cannot take the sum of metrics
in some averaging process to obtain a minimizing metric, the proof was obtained using Sullivan
purification \cite{Sull} and Rummler formula \cite{Rum}. In turn, with respect to a chosen metric,
a finite sum of basic vector fields will always preserve the positivity of transverse
divergence. We present a  swift proof of the result below using our tautness characterization.

\begin{proposition}[\'Alvarez L\'opez \cite{Al}]
Let $\pi :\tilde M\rightarrow M$ be a Riemannian finite covering of $M$. Then a Riemannian foliation
$\mathcal{F}$ on $M$ is taut if and only if the Riemannian foliation $\tilde {\mathcal{F}}:=\pi ^{*}\left( \mathcal{F}\right)$
on $\tilde M$ is taut.
\end{proposition}

\textit{Proof}. If $\mathcal{F}$ is non-taut, then a basic vector field $v$ with
$\mathrm{div}^Q {v}\ge 0$ and $\mathrm{div}^Q {v}>0$ at some point exists on $M$. Using the
local isometry $\pi $ we can consider on $\tilde M$ the pull-back $\pi^{*}\left( {v}\right)$
which will verify the positivity of transverse
divergence, so $\tilde {\mathcal{F}}$ will be non-taut. Conversely, assume
that $\tilde {\mathcal{F}}$ is non-taut and $\tilde{{v}}$ is a basic vector
field on $\tilde M$ such that $\mathrm{div}^Q\tilde{{v}}\ge 0$ and $\mathrm{div}^Q\tilde{{v}}>0$
at some point. Let $\Gamma $ be the group of deck
transformations of $\pi $. Then the finite sum $\sum_{\gamma \in \Gamma}\gamma ^{*}\left( \tilde{{v}}\right) $
represents a basic vector field on $\tilde M$ which also verify the divergence property and it can be projected
on $M$ to obtain a vector ${v}$ with the same features; consequently, $\mathcal{F}$ is non-taut.
\qed
\subsection{A new result: the characterization of taut foliations with dense leaves}

In the final part of the paper we briefly present a new consequence of the main result of Section 3.
We use Corollary \ref{taut characterization} to obtain a characterization of taut Riemannian foliations transversally oriented and
with dense leaves.

\begin{proposition}
If the foliation is transversally oriented and has dense leaves, then it is
taut if any and only if any globally defined basic vector field ${v}$
preserves the transverse volume form $\nu _Q$, \textrm{i.e.}
\begin{equation}
\mathcal{L}_v\nu _Q=0 \;.  \label{solenoidal}
\end{equation}
\end{proposition}

\begin{proof} From the defining property of the basic vector field (\ref{basic vector field}) it is easy to see that the differential operators $\mathcal{L}_{{v}}$ and $\theta _{{v}}$ agree on $\Omega _b\left( M\right) $ for ${v}$
basic. As the foliation has dense leaves, the smooth basic function $\mathrm{div}^Q {v}$ is constant, then in accordance with Corollary \ref{taut characterization} $\mathrm{div}^Q {v}$ vanishes on $M$. The conclusion comes
now from the relation (\ref{div_vol}), written for the transverse
distribution $Q$.
\end{proof}

\begin{remark}
Consequently, the image of any projectable vector field on a local transversal $T$ is \emph{solenoidal} and \emph{volume preserving} (see \textrm{e.g.} \cite[page 140]{Marsden-Ratiu}).
\end{remark}

\section*{Acknowledgements}

The author would like to thank J.~A. \'Alvarez L\'opez for helpful conversations.
He also would like to thank the referee for corrections and recommendations that improved the quality of the paper.
This work was partially supported by the UE grant FP7-PEOPLE-2012-IRSES-316338.


\begin{thebibliography}{99}


\bibitem{Al} J.~A. \'Alvarez L\'opez, J.~A.: The basic component of the mean
curvature of Riemannian foliations, Ann. Global Anal. Geom. 10 (1992), 179-194.

\bibitem{Al-Ko} J.~A. \'Alvarez L\'opez, Y.~A. Kordyukov, Long time
behaviour of leafwise heat flow for Riemannian foliations, Compositio
Math. 125 (2001), 129-153.

\bibitem{Bo-Ga}  C. Boyer, K. Galicki, 3-Sasakian manifolds, Surveys Differ. Geom. 7 (1999), 123-184.

\bibitem{Car}   Y. Carri\`ere, Flots riemanniens, Structures transverses des
feuilletages, Ast\'erisque 116 (1984), 31-52.

\bibitem{Do} D. Dom\'\i nguez, Finiteness and tenseness theorems for
Riemannian foliations, Amer. J. of Math. 120 (1998), 1237-1276.

\bibitem{Drag-Ornea} S. Dragomir,  L. Ornea, Locally Conformal K\"ahler
Geometry, Birkh\"auser, Boston, 1998.

\bibitem{Hae}  A. Haefliger, Some remarks on foliations with minimal leaves,
J. Differ. Geom. 15 (1980), 269-284.

\bibitem{I-S-V-V} A.~M. Ionescu, V. Slesar, M. Visinescu, G.~E. V\^\i lcu,
Transverse Killing and twistor spinors associated to the basic Dirac
operators. Rev. Math. Phys. 25 (2013), 1330011, 21 p.

\bibitem{Ko-No} S. Kobayashi,  K. Nomizu, Foundations of Differential
Geometry I. Interscience Publishers, New York, 1963.

\bibitem{Mas} X.~M. Masa, Duality and minimality in Riemannian foliations,
Comment. Math. Helv. 67 (1992), 17-27.

\bibitem{Mason} A. Mason, An application of stochastic flows to Riemannian
foliations, Houston J. Math. 26 (2000), 481-515.

\bibitem{Marsden-Ratiu}  J.~E. Marsden, T.~S. Ratiu, Introduction to
Mechanics and Symmetry: A Basic Exposition of Classical Mechanical Systems,
Springer, New York, 1999.

\bibitem{Mi-Ru-To} M. Min-Oo, E. Ruh, P. Tondeur, Vanishing theorems for
the basic cohomology of Riemannian foliations, J. Reine Angew. Math. 415 (1991), 167-174.

\bibitem{Noz}  H. Nozawa, Continuity of the \'Alvarez class under
deformations, J. Reine Angew. Math. 673 (2012), 125-159.

\bibitem{Noz-Prieto} H. Nozawa, J.~I. Royo-Prieto, Tenseness of Riemannian
flows, Ann. Inst. Fourier, 64 (2014) no. 4, 1419-1439.

\bibitem{Po} W.~A. Poor, Differential Geometric Structures, McGraw-Hill,
New York, 1981.

\bibitem{Re}  B. Reinhart, Foliated manifolds with bundle-like metrics, Ann.
Math. 69 (1959), 119-132.

\bibitem{Prieto-Sar-Wolak} J.~I. Royo-Prieto,  M. Saralegi-Aranguren, R. Wolak,
Cohomological tautness for Riemannian foliations, Russian J. Math. Phys. 16 (2009), 450-466.

\bibitem{Rum} H. Rummler, Quelques notions simples en g\'eom\'etrie
riemannienne et leurs applications aux feuilletages compacts, Comment. Math.
Helv. 54 (1979), 224-239.

\bibitem{Sl} V. Slesar, Spectral sequences and taut Riemannian foliations,
Ann. Glob. Anal. Geom. 32 (2007), 87-101.

\bibitem{Sull} D. Sullivan, A homological characterization of foliations
consisting of minimal surfaces, Comment. Math. Helv. 54 (1979), 218-223.

\bibitem{To} Ph. Tondeur, Geometry of Foliations, Birkh\"auser, Basel, 1997.

\bibitem{Vilcu} G.~E. V\^\i lcu, Ruled CR-submanifolds of locally conformal
K\"ahler manifolds, J. Geom. Phys. 62 (2012), 1366-1372.

\bibitem{Wal}  P. Walczak, An integral formula for a Riemannian manifold
with two orthogonal complementary distributions, Colloq. Math. 58 (1990), 243-252.


\end{thebibliography}
\end{document}